\documentclass[12pt]{amsart}
\usepackage{graphicx}
\usepackage[all]{xy}
\usepackage{amscd}
\usepackage{amssymb}
\usepackage{amsmath}
\usepackage{amsthm}
\usepackage{amsfonts}
\usepackage{graphics}
\usepackage{epsfig,psfrag}
\usepackage[english]{babel}
\usepackage{latexsym}
\usepackage{mathrsfs}

\newtheorem{theorem}{Theorem}[section]

\newtheorem{proposition}[theorem]{Proposition}
\newtheorem{question}[theorem]{Question}

\theoremstyle{definition}
\newtheorem{example}[theorem]{Example}

\newcommand{\RR}{{\mathbb R}}

\newcommand{\ZZ}{{\mathbb Z}}

\theoremstyle{remark}
\newtheorem{remark}[theorem]{Remark}

\begin{document}

\title[The minimal $n$ for equivariant embeddings  $B_g\to S^n$ is $n=2g-1$]
{The minimal dimension of a sphere with an equivariant embedding of the bouquet of $g$ circles is $2g-1$}

\author{Zhongzi Wang}
\address{Department of Mathematics Science, Tsinghua University, Beijing, 100080, CHINA}
\email{wangzz18@mails.tsinghua.edu.cn}




\begin{abstract} To embed the bouquet of $g$ circles $B_g$ into the $n$-sphere $S^n$ so that its full symmetry group action extends to an orthogonal actions on $S^n$, the minimal $n$ is $2g-1$. This answers a question raised by B. Zimmermann.
\end{abstract}

\date{}
\maketitle

In this note  graphs  are finite and connected and  group actions are faithful. Denote the $n$-dimensional Euclidean space and
orthogonal group by $\RR^n$ and $O(n)$, and the unit sphere in $\RR^{n+1}$ by $S^n$.
 
Graphs belong to  a most  familiar geometric subject to us and the symmetries of graphs  is a big and active research area. 

For a given graph $\Gamma$ 
with an action of a finite group $G$, an embedding $e: \Gamma \to S^n$ is $G$-equivariant, if there is a subgroup $G \subset O(n+1)$ so that $G$ acts on the pair $(S^n, e(\Gamma))$
satisfying  $g(e(x))=e(g(x))$ for any $x\in \Gamma$ and $g\in G$.
In the definition of $G$-equivariant embedding, one may  replace $(S^n, O(n+1))$ by $(\RR^n, O(n))$.
See Remark \ref{definitions} for  relations between the two definitions.

The existence of $G$-equivariant embedding for any pair $(G,\Gamma)$ is known,  see \cite{Mos} and \cite{Pa}.
On the other hand, for given $\Gamma$ and $G$, 
usually it is difficult to find the minimal $n$ so that there is an $G$-equivariant embedding $\Gamma\to \RR^n$.

\begin{figure}[h]
\includegraphics{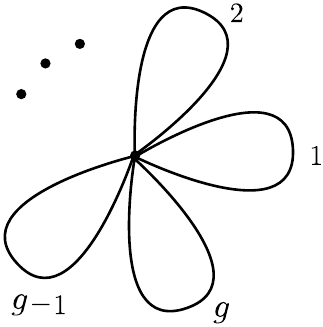}
\caption{A bouquet of $g$ circles: }
\end{figure}

The most significant finite group associated to a graph $\Gamma$ is  its full symmetry group $\text{Sym}(\Gamma)$.
The simplest and also most symmetric graph is the bouquet of $g$ circles (the graph with one vertex and $g$ closed  edges, see Figure 1),
denoted as $B_g$. $B_g$ is the simplest in the sense that $B_g$ has only one vertex. $B_g$ is the most symmetric in the  sense that:
for $g\ge3$, the maximum order of finite 
group action on a hyperbolic graph of rank $g$ (defined as its fundamental group rank) is $2^gg!$, which 
is  realized uniquely by the action of  $\text{Sym}(B_g)$ on $B_g$ \cite{WZ}.
A graph $\Gamma$ is hyperbolic if it has rank $>1$  and  has no free edge. See Example 0.6 for the rank 2 case.

We will simply use equivariant embedding of $\Gamma$ for $\text{Sym}(\Gamma)$-equivariant embedding of $\Gamma$.

Recently B. Zimmermann asked a very interesting question (as commented by the Math. Review of the paper).

\begin{question} (\cite{Zi})
What is the minimal $n$ so that there is an equivariant embedding $B_g\to S^n$? 
\end{question}

Zimmermann has pointed that $n\ge g-1$  from a group theoretical reason.
He also described an equivariant embedding of $B_g$ into $S^{2g}$. So he concluded the minimal $n$ is between $g-1$ and $2g$ 
\cite{Zi}. 
We will answer this question.

\begin{theorem}\label{main}
The minimal $n$ so that there is an equivariant embedding $B_g\to S^n$ is $2g-1$.
\end{theorem}

Theorem \ref{main} follows from the next two propositions.

\begin{proposition}\label{main1}
Suppose there is an equivariant embedding of $B_g$ into $S^n$, then  $n\ge 2g-1$.
\end{proposition}

\begin{proposition}\label{main2}
 There are equivariant embeddings of $B_g$ into $S^{2g-1}$ for each $g>1$.
\end{proposition}

 \begin{proof}[{Proof of Proposition \ref{main1}}] 
We need  first recall a careful description of the action of $\text{Sym}(B_g)$ on $B_g$. 
Denote the vertex of $B_g$ by $v$ and the $g$ circles by $C_1,..., C_g$, see Figure 2. 
From now on we consider that each circle $C_i$ is oriented. 
It is known that $\text{Sym}(B_g)$ is the semi-direct product  $(\ZZ_2)^g\ltimes S_g$, where $S_g$ is the permutation group of $g$ elements,
which permutes those $g$ oriented circles,  and the normal subgroup $(\ZZ_2)^g$ acts by orientation reversing involutions  of those $g$ oriented circles.

\begin{figure}[h]
\includegraphics{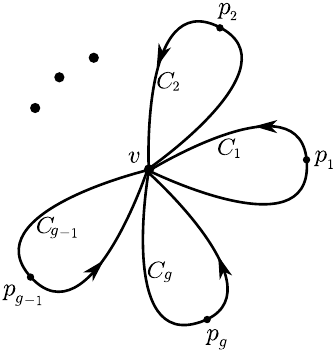}
\caption{Detailed description $\text{Sym}(B_g)$ action on $B_g$}
\end{figure}

More precisely, for each $\sigma\in  S_g$, denote the corresponding action on $B_g$ by $\tau_\sigma$,  which sends $C_i$ to $C_{\sigma(i)}$ and preserves the orientations.
Denote the element $\rho_i$  to be the element in $(\ZZ_2)^g\subset \text{Sym}(B_g)$, which is the orientation reversing involution on $C_i$
and is the identity on the remaining $C_j$, for $j\ne i$. Now $\rho_i$ has two fixed points, one is $v$, another we denote as $p_i$. 
It is clear that 
$$\tau_\sigma\rho_i\tau_{\sigma^{-1}}=\rho_{\sigma(i)}.    \qquad (1)$$
By (1) we have

$$\rho_{\sigma(i)}\tau_{\sigma}(p_i)=\tau_\sigma\rho_i\tau_\sigma^{-1}\tau_\sigma(p_i)=\tau_\sigma\rho_i(p_i)=\tau_\sigma(p_i).$$
So we have

$$\rho_{\sigma(i)}\tau_{\sigma}(p_i)=\tau_\sigma(p_i).$$

That is,  $\tau_\sigma(p_i)$ is the fixed point of $\rho_{\sigma(i)}$ on $C_{\sigma(i)}-\{v\}$. Since the fixed point of $\rho_{\sigma(i)}$ on $C_{\sigma(i)}-v$ is unique, 
we have 
$$\tau_\sigma(p_i)=p_{\sigma(i)}.\qquad (2)$$

Suppose now we have an equivariant embedding $B_g\subset  S^n$. That is to say we have embedded $\text{Sym}(B_g)$ into $O(n+1)$
which acts on the pair $(S^n, B_g)$ and the restriction on $B_g$ is the action of $\text{Sym}(B_g)$ described as above.
Since $S^n$ is the unit sphere of $\RR^{n+1}$ and $O(n+1)$ is the orthogonal transformation group  of $\RR^{n+1}$,
the action of $\text{Sym}(B_g)$ on $S^n$ extends to  a  $\text{Sym}(B_g)$-action on $\RR^{n+1}$ as an  orthogonal transformation group.

Consider each point in $\RR^{n+1}$ as a vector in $\RR^{n+1}$. Since $v, p_1, p_2, ..., p_g$ are in the unit sphere, each of them is a non-zero vector.
Let $$V=\left<v, p_1, p_2, ..., p_{g}\right>$$  be the subspace of $\RR^{n+1}$ spanned by the vectors $v, p_1, p_2, ..., p_g$. 

By (2), the $S_g$-action on $B_g$ permutes those vectors $p_1, p_2, ..., p_g$ and fixes $v$. Hence $S_g$ acts faithfully on $V$,
that is $\tau_\sigma|_V=\text{id}|_V$ if and only if $\sigma$ is the identity in $S_g$.

Now we show that the vectors $v, p_1,..., p_{g-1}$ are linearly independent:

Suppose $a, a_1,...,a_{g-1}$ are real numbers such that 

$$ av+\sum_{i=1}^{g-1}a_ip_i=0. \qquad (3)$$

For each $j\in \{1,2,..., g-1\}$, recall that $(j, g)\in S_g$ exchanges $j$ and $g$, and fixes all the remaining $i$,  and
$\tau_{(j,g)}$ corresponds to the element $(j,g)\in S_g$.
Since $\tau_{(j,g)}$ is an orthogonal transformation, $\tau_{(j,g)}$ is linear. Applying $\tau_{(j,g)}$ to (3), we have 
 
$$0=\tau_{(j,g)}(av+\sum_{i=1}^{g-1}a_ip_i)=a\tau_{(j,g)}(v)+\sum_{i=1}^{g-1}a_i\tau_{(j,g)}(p_i). \qquad (4)$$

Since $v$ is the common fixed point of $S_g$ and $\tau_{(j,g)}(p_i)=p_i$ if $j\ne i$ and $\tau_{(j,g)}(p_j)=p_g$, 
we have further

$$av+\sum_{i=1, {i\ne j}}^{g-1} a_ip_i + a_j p_g=0. \qquad (5)$$

On the other hand by (3) and some elementary transformations we have

  $$av+ \sum_{i=1, {i\ne j}}^{g-1} a_ip_i + a_j p_g=  av+\sum_{i=1}^{g-1} a_ip_i-a_j p_j + a_j p_g=a_j(p_g-p_j). \qquad (6)$$
  
  Combine (5) and (6), we have 
  $$a_j(p_g-p_j)=0.$$
  Since $B_g\subset \RR^{n+1}$ is an embedding, and $j\ne g$, we have 
  $$p_g-p_j\ne 0.$$
  Hence $a_j=0$. Since $j$ can be any element in $\{1,2,..., g-1\}$. Hence $a_j=0$ for each $j\in \{1,2,..., g-1\}$.
  By (3) we have 
  $$av=0.$$
  Since $v\in S^n$, $v\ne 0$, we have $a=0$.
  That is (3) implies that 
  $$a=a_1=a_2=...=a_{g-1}=0, \qquad (7)$$
  that is  $v, p_1,..., p_{g-1}$ are linearly independent.
  As a conclusion we have 
  
$$\text{dim} V\ge g.\qquad (8)$$

Let $V^\perp$ be orthogonal complement of $V$ in $\RR^{n+1}$. We have 
$$\RR^{n+1}=V\oplus V^\perp. \qquad (9)$$

For any $i\in \{1,2,..., g\}$,   $\rho_i$ is orthogonal, so is linear. Moreover each point in $\{v, p_1, p_2, ..., p_{g}\}$ is a fixed point of $\rho_i$.
Therefore  $\rho_i$
is the identity on $V$, the subspace spanned by $\{v, p_1, p_2, ..., p_{g}\}$. Further more we conclude that the 
action of the subgroup $(\ZZ_2)^g$ on $V$ is trivial. 
In particular $V$ is an invariant subspace of the $(\ZZ_2)^g$-action.

Since the $(\ZZ_2)^g$-action on $\RR^{n+1}$ is orthogonal and $V$ is invariant under the $(\ZZ_2)^g$-action, $(\ZZ_2)^g$ acts
orthogonally on the orthogonal 
complement $V^\perp$ of $V$.
Since the $(\ZZ_2)^g$-action is trivial on $V$ and is faithful on $\RR^{n+1}$,  the $(\ZZ_2)^g$-action must be  faithful on $V^\perp$
by (9).

Let  $q$ be the dimension of $V^{\perp}$, and consider $O(q)$ as $q$ by $q$ orthogonal  matrix group.
Then $(\ZZ_2)^g$-action on $V^\perp$ is by a matrix group in $O(q)$. 
Let $\Omega_q$ be the set of all $q$ by $q$   diagonal matrices whose entries on the diagonal of the matrices are either $1$ or $-1$.
Then there are exactly $2^q$  matrices in $\Omega_q$.

Each element in $(\ZZ_2)^g$ is of order 1 or 2, therefore it can be diagonalized, indeed into $\Omega_q$.
By linear algebra, a finite number of commuting diagonalizable matrices can be diagonalized simultaneously.  Since $(\ZZ_2)^g$ is a
finite abelian group,
we may assume that under a basis of $V^\perp$, all matrices in $(\ZZ_2)^g$ are diagonalized, therefore are elements in $\Omega_g$.
Since $(\ZZ_2)^g$ has $2^g$ elements, $(\ZZ_2)^g$ acts faithfully on $V^\perp$, the image of $(\ZZ_2)^g$ in $\Omega_g$ must be also $2^g$ 
pairwise 
different elements.
Therefore $2^g\le 2^q$, that is 
$$g\le q=\text {dim} V^\perp. \qquad (10)$$
By (8),  (9) and (10) we have 
$$n+1=\text {dim} V+\text {dim} V^\perp\ge 2g.$$
That is 
$$n\ge 2g-1.$$
We finish the proof.   \end{proof}
  
\begin{proof}[{Proof of Proposition \ref{main2}}] 
We will contruct two different equivariant embeddings.
The first one is
an  equivariant embedding $B_g\to \RR^{2g-1}$,  then the required $G$-equivariant embedding $B_g\to S^{2g-1}$ can be obtained by one point compactification
via the inverse of stereographic projection $p: S^{2g-1}\to \RR^{2g-1}$. 
The second one is an equivariant embedding of $B_g\to \RR^{2g}$, where $B_g$ stays equivariantly in the unit sphere $S^{2g-1}$. 
 
{\bf The first construction.} Before we give a general construction, we 
  like to  provide a visible equivariant embedding of $B_2\subset \RR^{3}$ in the familiar dimension in which we live.

Fix a standard $xyz$-coordinate system of $\RR^3$.
Let $C_1$ be the unit circle in the lower-half $zx$-plane which is tangent to the $x$-axis at the origin, and 
$C_2$ be the unit circle in the upper-half $yz$-plane which is tangent to the $y$-axis at the origin, as shown in Figure 3.
Then $C_1\bigcup C_2$ provides  an embedding $B_2\subset \RR^{3}$.
The action of  $\text{Sym}(B_2)=(\ZZ_2)^2\ltimes S_2$ is also realized by a subgroup of $O(3)$:
where the inversion of $C_1$ is given by the reflection about the $yz$-plane,  the inversion of $C_2$ is given by the 
reflection about the $zx$-plane, and the permutation of $C_1$ and $C_2$ is given by the $\pi$-rotation around a diagonal $L$ in the $xy$-plane.
This provides an equivariant embedding.

\begin{figure}[h]
\includegraphics{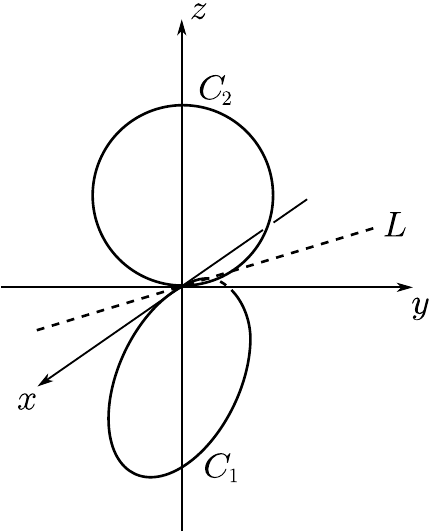}
\caption{An equivariant embedding of $B_2$ into $\RR^3$}
\end{figure}

Then we provide the equivariant embedding of $B_g\subset \RR^{2g-1}$ for $g>2$.

Let $V_1=\RR^{g-1}$ and $v_1, v_2,...,v_g$ be the vertices of a regular $(g-1)$-dimensional simplex $\Delta_{g-1}$ centered at the origin $O_1$ of 
$V_1$ with the length $|O_1v_i|=1$. Then $S_g$, the full symmetry group of $\Delta_{g-1}$, 
 is a subgroup of $ O(g-1)$.
Let $V_2=\RR^g$, and $e_1,e_2, ... , e_g$ be the standard orthogonal basis of $V_2$.
Let $V=V_1\times V_2$, then $V=\RR^{2g-1}$,  and $(v_i, e_i)$ is a standard orthogonal basis of the Euclidean plane $E_i$ spanned by $v_i$ and $e_i$, $i\in \{1,...,g\}$.

Define the isometrical embedding  $\iota_i: S^1 \to V$  by 
$$\iota_i(cos\theta, sin\theta)=(1+cos\theta)v_i+sin\theta e_i, \qquad (11)$$
where $0\le \theta < 2 \pi$ is the parameter on the unit circle.

Let $\iota_i(S^1)=C_i$, $i\in \{1,...,g\}$.  Clearly the origin $O$ of $V$ belongs to each $C_i$. Moreover $C_i$ belongs to $E_i$, and $E_i\cap E_j=O$
for $i\ne j$, so $C_i\cap C_j=O$ for $i\ne j$. Hence
$\bigcup_{i=1}^g C_i$ is a bouquet of $g$ circles embedded in $V$ with the origin $O$ as the vertex, still denoted as $B_g$. 
 
For each $i\in \{1,2,...,g\}$ and each $\sigma\in S_g$, we define  $\iota_i, \tau_\sigma \in O(2g-1)$ as follows    
$$\iota_i(v_j)=v_j\,\,  \text {for all $j$},  \,\, \iota_i(e_j)=e_j \,\,  \text {for $j\neq i$}; \,\,  \iota_i(e_i)=-e_i. \qquad (12)$$

$$ \tau_\sigma(e_i)=e_{\sigma(i)}, \,\, \tau_\sigma(v_i)=v_{\sigma(i)}. \qquad (13)$$

By (11), (12) and (13), one can check directly

(i) $\iota_i$ is the identity on $C_j$ for $j\ne i$ and is an orientation reversing involution on $C_i$.

(ii)  $\tau_\sigma(C_i)=C_{\sigma(i)}$ and $\tau^{-1}_\sigma(C_i)=C_{\sigma^{-1}(i)}$. 


Let $$H=\left< \iota_i, \tau_{\sigma}: \,\, 1\le i \le g, \,\, \sigma \in S_g \right>.$$
Then $H$ is a subgroup of $\text{Sym}(B_g)$, and each $h\in H$ has the form 
$$h=\tau_\sigma \prod_{i=1}^g \iota_i^{\alpha_i},$$
where $\alpha_i=0$ or 1, since $\left< \iota_i, \,\, 1\le i \le g \right>$ is a normal subgroup of $H$.
Clearly $H$ acts faithfully on $B_g$.
Moreover if  $h\in H$ is the identity on $B_g$, $h$ must fix all $v_i$ and $e_i$, which implies that $h$ is the identity on $V$.
So $H$ also acts faithfully on $V$. 
The restriction of $H$ on $B_g$ is $\text{Sym}(B_g)$, since it exhausts all symmetries of $\text{Sym}(B_g)$.

Therefore we get an equivariant embedding $B_g\to \RR^{2g-1}$.

{\bf The second construction.} View $\RR^{2g}$ as the product  
$$\RR^{2g}=\RR^2_1\times \RR^2_2\times ... \times \RR^2_g,$$
where each $\RR^2_i$ is an Euclidean plane with standard coordinate system $(x_i, y_i)$, $i=1,2,..., g$. 
Then $(x_1, y_1, x_2, y_2, ... , x_g, y_g)$ provides a standard coordinate system of the Euclidean space $\RR^{2g}$.
Let $C_i$ be the circle in $\RR^2_i\subset \RR^{2g}$ given by 
$$C_i=\{ x_i^2+y_i^2=\frac 1g, \,\, x_j=\frac 1{\sqrt g}, y_j=0\,\, \text{for remaining}\,\,  j\ne i.\} \qquad (13)$$ 
with anti-clockwise orientation in the $x_iy_i$-plane.

Now we make three observations:

(a) Clearly the point $$v=(\frac 1{\sqrt g}, 0, \frac 1{\sqrt g}, 0, ..., \frac 1{\sqrt g},0)\in \RR^{2g}$$ belongs to each $C_i$, and  each pair $C_i$ and $C_j$ intersect only at $v$ for $i\ne j$.
Therefore the union $\bigcup_{i=1}^g C_i$ provides an embedding of $B_g\to \RR^{2g}$.

(b) For each $i\in \{1,2,...,g\}$ and each $\sigma\in S_g$, we define  $$\iota_i, \tau_\sigma :\RR^{2g}\to \RR^{2g}$$ as below:
$$\iota_i(x_1, y_1, ... , x_{i-1}, y_{i-1}, x_i, y_i, x_{i+1}, y_{i+1},..., x_g, y_g)$$
$$=(x_1, y_1, ... , x_{i-1}, y_{i-1}, x_i, -y_i, x_{i+1}, y_{i+1},..., x_g, y_g).$$

$$\tau_\sigma(x_1, y_1, ... ,  x_i, y_i,..., x_g, y_g)$$
$$=(x_{\sigma(1)}, y_{\sigma(1)}, ... , , x_{\sigma(i)}, y_{\sigma(i)}, ,..., x_{\sigma(g)}, y_{\sigma(g)})$$

Clearly $\iota_i, \tau_\sigma\in O(2g)$,  $\iota_i$ is the identity on $C_j$ for $j\ne i$ and is an orientation reversing involution on $C_i$. 
and $\tau_\sigma(C_i)=C_{\sigma(i)}$. Moreover the group
 $$H=< \iota_i, \tau_{\sigma}: \,\, 1\le i \le g, \,\, \sigma \in S_g>$$
 is isomorphic to $(\ZZ_2)^g\rtimes S_g \cong \text{Sym}(B_g)$ and  acts faithfully on  the pair $(\RR^{2g}, B_g).$

(c)  For each $w\in B_g$, $w\in C_i$ for some $i$. By (13) 
we have $$|w|^2 = \sum_{j=1}^g(x_j^2+y_j^2)=\sum_{j=1, j\ne i}^g(x_j^2+y_j^2)+(x_i^2+y_i^2)$$
$$=\sum_{j=1, j\ne i}^g((\frac 1{\sqrt g})^2+0^2)+\frac 1g=1.$$
That is to say $B_g$ stays in the unit sphere $S^{2g-1}\subset \RR^{2g}$.

By the conclusion of (b),  the given $B_g\subset S^{2g-1}$ is an equivariant embedding.
\end{proof}

\begin{remark}\label{definitions} For given pair $(G, \Gamma)$, 
as we explained in the proof of  Proposition \ref{main2}, each $G$-equivariant embedding $\Gamma\to \RR^n$  provides a $G$-equivariant embedding $\Gamma\to S^n$. 
However a $G$-equivariant embedding  $\Gamma\to S^n$ does not guarantee a $G$-equivariant embedding  $\Gamma\to R^n$, see the example below. 
\end{remark}

\begin{example}
The maximum order of finite 
group action on a hyperbolic graph of rank $2$  is $12$, which 
is  realized uniquely by the action of  $\text{Sym}(M_3)=S_3\ltimes \ZZ_2$ on $M_3$ \cite{WZ}, 
where $M_3$ is the graph with two vertices joined by three edges, 
$S_3$ permutes three edges and fixes each vertex, and $\ZZ_2$ is an orientation-reversing involution on each edge.
There is an equivariant embedding $M_3\to S^2$, which is indicated by the left side of Figure 4.

\begin{figure}[h]
\includegraphics{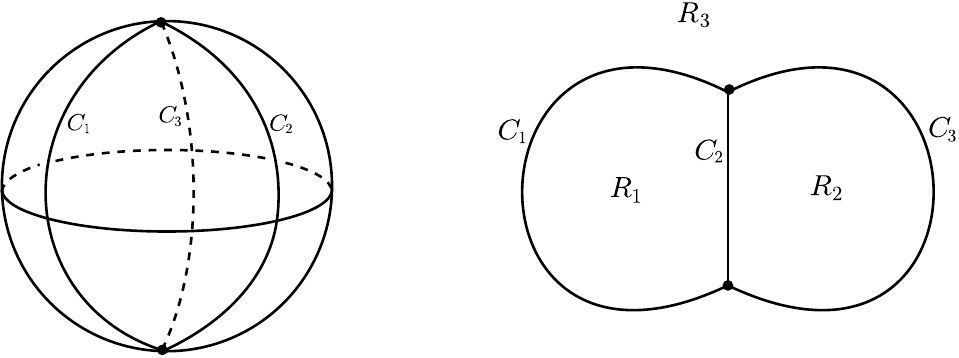}
\caption{An equivariant embedding $M_3\to S^2$ }
\end{figure}

There is no equivariant embedding $M_3\subset  \RR^2$. A quick proof uses Jordan curve theorem:
 For any embedding $ M_3\to \RR^2$,  $M_3$ divides $\RR^2$ into   two  bounded regions $R_1$, $R_2$, and one unbounded
 region $R_3$,  see the right side of Figure 4.
Then $C_2$, the unique edge neighboring two bounded regions, must be invariant under any $\tau\in O(2)$, therefore the embedding is  not equivariant.
\end{example}

\begin{remark}
A related question is  when  a finite group action $G$ on a graph $\Gamma$ can be $G$-equivariantly  embedded into  
$\RR^3$ ($S^3$), see \cite{FNPT}, 
\cite{WWZZ2}. Similar question for  surfaces is also addressed, see \cite{CC},  \cite{WWZZ1}.
\end{remark}

{\bf Acknowledgements.} We thank the referees for their suggestions which enhance the paper.

\end{document}